\documentclass [14pt]{amsart}
\usepackage{amsmath}
\usepackage{amssymb}
\usepackage{amscd}
\usepackage{epsfig}
\usepackage{amsxtra}
\usepackage{mathabx}
\usepackage{MnSymbol}
\usepackage{scalerel,stackengine}
\usepackage{xcolor}
\usepackage{epsfig}
\usepackage{tikz-cd}
\usepackage{hyperref}
\usepackage{tcolorbox}
\usepackage[pagewise]{lineno}
\allowdisplaybreaks

\newcommand{\rank}{\mbox{\rm rank}}
\newcommand{\aprank}{\mbox{\rm ap-rank}}
\newcommand{\RR}{{\mathbb R}}

\newcommand{\QQ}{{\mathbb Q}}
\newcommand{\ZZ}{{\mathbb Z}}

\newcommand{\NN}{\mathbb N}
\newcommand{\KK}{\mathbb K}

\newcommand{\cL}{{\mathcal L}}

\newcommand{\oplam}{\mbox{\Large $\curlywedge$}}

 \newtheorem{theorem}{Theorem}[section]
 \newtheorem{lemma}[theorem]{Lemma}
 \newtheorem{proposition}[theorem]{Proposition}
 \newtheorem{corollary}[theorem]{Corollary}

 \newtheorem{quest}[theorem]{Question}
 \newtheorem{definition}[theorem]{Definition}
 \newtheorem{remark}[theorem]{Remark}
 \theoremstyle{definition}
 \newtheorem{example}[theorem]{Example}

\begin{document}
\title[Higher dimensional progressions]{On higher dimensional arithmetic progressions in Meyer sets}
\author{Anna Klick}
\address{Department of Mathematical Sciences, MacEwan University \newline
\hspace*{\parindent}10700 -- 104 Avenue, Edmonton, AB, T5J 4S2, Canada}
\email{klicka@mymacewan.ca}

\author{Nicolae Strungaru}
\address{Department of Mathematical Sciences, MacEwan University \newline
\hspace*{\parindent}10700 -- 104 Avenue, Edmonton, AB, T5J 4S2, Canada\\
and \\
\newline
\hspace*{\parindent}Institute of Mathematics ``Simon Stoilow'' \newline
\hspace*{\parindent}Bucharest, Romania}
\email{strungarun@macewan.ca}
\urladdr{http://academic.macewan.ca/strungarun/}

\begin{abstract}In this paper we study the existence of higher dimensional arithmetic progression in Meyer sets. We show that the case when the ratios are linearly dependent over $\ZZ$ is trivial, and focus on arithmetic progressions for which the ratios are linearly independent. Given a Meyer set $\Lambda$ and a fully Euclidean model set $\oplam(W)$ with the property that finitely many translates of $\oplam(W)$ cover $\Lambda$, we prove that we can find higher dimensional arithmetic progressions of arbitrary length with $k$ linearly independent ratios in $\Lambda$ if and only if $k$ is at most the rank of the $\ZZ$-module generated by $\oplam(W)$. We use this result to characterize the Meyer sets which are subsets of fully Euclidean model sets.
\end{abstract}

\maketitle

\section{Introduction}
The Nobel Prize discovery of quasicrystals \cite{She} sparked many questions regarding the nature of solids with long-range aperiodic order. This discovery led to establishment of a new area of mathematics, the area of Aperiodic Order. The goal of this new field is to study objects which show long range order, but they are not necessarily periodic. \\

The best mathematical models for point sets which show long range order and are typically aperiodic were discovered in the earlier pioneering work of Y. Meyer \cite{Meyer}, and have been popularized in the field by R.V. Moody \cite{MOO,Moody} and J. Lagarias \cite{LAG1,LAG}. Called model sets, they are constructed via a cut-and-project scheme, a mechanism which starts with a lattice $\cL$ in a higher dimensional space which sits at an ``irrational slope" with respect to the real space $\RR^d$, cuts a strip around the real space $\RR^d$ of bounded width $W$ (called the ``window") and projects it on the real space $\RR^d$ (see Definition~\ref{def:Cps} for the exact definition). Under various weak conditions, the high order present in the lattice $\cL$ shows in the resulting model set, typically via a clear diffraction diagram, for example see \cite{BM,Hof1,LSS,CR,Martin2} just to name a few.

\medskip

Meyer sets are relatively dense subsets of model sets. As subsets of model sets they inherit part of the high order present in the former, which is evident in their characterisation via discrete geometry, harmonic analysis and algebraic properties \cite{LAG1,Meyer,MOO}. While they typically have positive entropy, and hence are usually not pure point diffractive (\cite{BLR}, compare \cite{BG2} for a discussion), they still show a highly ordered diffraction diagram \cite{NS2,NS11,NS20,NS21} with a relatively dense supported pure point spectrum \cite{NS1,NS2,NS5,NS21}.

\medskip

In \cite{KST} we showed a different type of order in Meyer sets in the form of the existence of arbitrarily long arithmetic progressions. More precisely, we proved that given a Meyer set $\Lambda \subseteq \RR^d$, for all $N \in \NN$ there exists some $R>0$ such that the set $\Lambda \cap B_R(x)$ contains an arithmetic progression of length $N$ for all $x \in \RR^d$. Moreover, we showed that van der Waerden type theorems hold in Meyer sets. More recently, related results have been investigated in \cite{AGNS,NAL}.

\smallskip

Consider now a Meyer set $\Lambda \subseteq \RR^d$. While $\Lambda$ spreads relatively dense in all directions of $\RR^d$, any arithmetic progression is one dimensional, and hence only gives partial information about the structure of Meyer sets. This suggest that one should look for higher dimensional arithmetic progressions, which is the goal of this paper. By a $m$-dimensional arithmetic progression we understand a set of the form
\[
A=\{ s+c_1r_1+\dots+c_mr_m : 0 \leq c_j \leq N_j \,, \forall 1 \leq j \leq m \}
\]
for some fixed $s,r_1,\dots,r_m \in \RR^d$ and $N_1,\dots,N_m \in \NN$. The elements $r_1,\dots,r_m$ are called the ratios and $\vec{N}=(N_1,\dots,N_m)$ is the vector length of the progression. The arithmetic progression is proper if all the elements $s+c_1r_1+\dots+c_mr_m$ are distinct.

By a standard application of the Chinese Remainder Theorem we show in Prop.~\ref{prop1} that for all $n \in \NN$ and $\vec{N} \in \NN^m$, every Meyer set contains a proper $n$-dimensional arithmetic progression of length $\vec{N}$. While the arithmetic progression is proper, every pair of ratios is linearly dependent over $\ZZ$ and hence the arithmetic progression is a subset of a 1-dimensional affine $\QQ$-space.

To make the question more interesting and meaningful, we add the extra condition that the ratios are linearly independent over $\ZZ$ (or equivalently over $\QQ$). Given a fully Euclidean model set $\oplam(W)$ in a CPS $(\RR^d, \RR^m, \cL)$, we show in Thm.~\ref{thm:ap rank model} that $\oplam(W)$ has $n$-dimensional arithmetic progressions of arbitrary length with linearly independent ratios if and only if $n \leq d+m$.

Next, for any Meyer set $\Lambda$, it is well known that there exists some fully Euclidean model set $\oplam(W)$ in some CPS $(\RR^d, \RR^m, \cL)$ and a finite set $F \subseteq \RR^d$ such that
\begin{equation}\label{eq121}
\Lambda \subseteq \oplam(W)+F \,.
\end{equation}
We show in Thm.~\ref{theo:meyer AP} that $\Lambda$ has $n$-dimensional arithmetic progressions of arbitrary length with linearly independent ratios if and only if $n \leq d+m$. This implies that, while in general there exists multiple fully Euclidean model sets $\oplam(W)$ and finite sets $F$ such that \eqref{eq121} holds, $m$ must be the same for all these model sets.

We complete the paper by answering the following question.
\begin{quest} Which Meyer sets $\Lambda \subseteq \RR^d$ are subsets of fully Euclidean model sets?
\end{quest}
To our knowledge no characterisation for these sets (which we call fully Euclidean Meyer sets) is known so far. We show in Thm.~\ref{fully euc meyer} that a Meyer set $\Lambda \subseteq \RR^d$ is a subset of a fully Euclidean model set if and only if has $n$-dimensional arithmetic progressions of arbitrary length with linearly independent ratios, where $n$ is the rank of the $\ZZ$-module generated by $\Lambda$. The characterisation is of number theory/combinatorics origin, emphasizing once again the nice order present in Meyer sets and model sets.\\

The paper is structured in the following way. In Section~\ref{s2} we give basic definitions and prove a higher-dimensional version of the van der Waerden's theorem \cite{vdW}. We prove, in Section~\ref{S3}, that  Meyer sets $\Lambda \in \RR^d$ contain arithmetic progression of arbitrary length and dimension, albeit with linearly dependent ratios. In Section~\ref{S4} we establish both the existence of arithmetic progression with linearly independent ratios, and a higher-dimensional van der Waerden-type result for fully Euclidean model sets. In Section~\ref{S5}, we extend these results to arbitrary Meyer sets in $\RR^d$. We complete the paper by characterising the fully Euclidean Meyer sets.\\

\section{Preliminaries}

In this section we review the basic definitions and results needed in the paper.

\subsection{Finitely generated free $\ZZ$-modules}\label{s2}

We start by recalling few basic results about finitely generated free $\ZZ$-modules. First recall \cite[Thm.~VIII.4.12]{AlgBook} that if $M$ is a free $\ZZ$-module, then all the bases of $M$ have the same cardinality. The common cardinality
of these bases is called the \textbf{rank} of $M$ and is denoted by $\rank(M)$. Also, a finitely generated $\ZZ$-module is free if and only if it is torsion free \cite[Thm.~12.5]{DF}.

\medskip

Next, let us recall the following two results which we will use a few times in the paper.

\begin{lemma} \cite[Thm.~VIII.6.1]{AlgBook}\label{lem:subsmodfree is free} Let $M$ be a free $\ZZ$-module of rank $k$. If $N$ is a submodule of $M$, then $N$ is free and
\[
\rank(N) \leq \rank(M) \,.
\]
\end{lemma}

In particular, Lemma~\ref{lem:subsmodfree is free} implies the following result.

\begin{corollary}\label{cor1} Let $M$ be a free $\ZZ$-module of rank $k$. If $v_1,\dots,v_m \in M$ are linearly independent over $\ZZ$ then $m \leq k$.
\end{corollary}
\begin{proof}
$v_1,\dots,v_m$ are a basis for the submodule of $M$ generated by $\{ v_1,\dots,v_m \}$. The claim follows from Lemma~\ref{lem:subsmodfree is free}.
\end{proof}

Next, let us recall the following result about $\ZZ$-submodules of the same rank as the full module.

\begin{lemma}\label{nM N} Let $M$ be a finitely generated free $\ZZ$-module and let $N$ be a submodule of $M$. If
\[
\rank(M)= \rank(N)
\]
then, there exists some positive integer $n$ such that
\[
\{ nv :v \in M \}=: n\cdot M \subseteq N \,.
\]
\end{lemma}
\begin{proof}
This follows from \cite[Thm.~VIII.6.1]{AlgBook}.
\end{proof}

\medskip

We complete the section by proving the following simple result. This result is surely known, but we could not find a good reference for it.

\begin{lemma}\label{lem:12} Let $M$ be a finitely generated free $\ZZ$-module with $k=\rank(M)$ and let $S$ be a generating set for $M$. Then, there exists $k$ linearly independent elements $v_1,\dots,v_k \in S$.
\end{lemma}
\begin{proof}

By a standard application of Zorn's lemma, there exist some $S'=\{ v_1,\dots, v_m\} \subseteq S$ which is a maximal linearly independent subset. Let $N$ be the submodule of
$M$ generated by $S'$. Then, $\rank(N)=m$. To complete the proof, we show that $m=k$. By \cite[Thm.~12.4]{DF}, there exists a basis $y_1,\dots, y_k$ of $M$ and $n_1,\dots, n_m \in \ZZ$ such that $n_1y_1,\dots,n_my_m$ is a basis for $N$.

Now assume by contradiction that $m <k$. Since $S$ spans $M$, there exists elements $x_1,\dots, x_l \in S$ and $k_1,\dots,k_l \in \ZZ$ such that
\[
y_k =k_1x_1+\dots+k_lx_l \,.
\]

Next, note that for each $1 \leq j \leq l$ we either have $x_j \in S'$ or $S' \cup x_j$ is linearly dependent. In both cases, there exists some non-zero $f_j \in \ZZ$
such that $f_j x_j \in N$. Let $f =f_1\cdot \ldots \cdot f_l$. Then $f$ is a non-zero integer and
\[
fy_k \in N \,.
\]
Since $n_1y_1,\dots,n_my_m$ is a basis for $N$, $fy_k$ can be written as a linear combination of $n_1y_1,\dots,n_my_m$. As $f \neq 0$, this gives that $y_1,\dots,y_m,y_k$ are linearly dependent over $\ZZ$, which is a contradiction. Thus, $m=k$.
\end{proof}

\medskip

Note that one can alternately prove the above lemma by embedding $M$ into a $\QQ$-vector space and looking at the subspace generated by $S'$.

\bigskip

\subsection{Higher dimensional arithmetic progressions}

In this section we look at higher dimensional arithmetic progressions. Let us start with the following definition.

\begin{definition} A \textbf{higher dimensional arithmetic progression} is a set of the form
\[
A := \{s+ c_1r_1 + c_2r_2+\dots+c_nr_n : 0\leq c_1 \leq N_1, 0 \leq c_2 \leq N_2,\dots, 0 \leq c_n \leq N_n\} \,
\]
for some fixed vectors $s, r_1,\dots,r_n \in \RR^d$ and some arbitrarily fixed natural numbers $N_1,\dots,N_n \in \mathbb{N}$. If the elements in $A$ are distinct, then the progression is called \textbf{proper}. We will call $n$ the \textbf{dimension} of the projection and $\vec{N}=(N_1,\dots,N_n)$ the \textbf{vector length} of the progression. The \textbf{rank} of the projection is the rank of the $\ZZ$-module generated by the ratios $\{ r_1,\dots,r_n \}$. We say that the arithmetic progression is an \textbf{li-arithmetic progression} if the ration $r_1,\ldots,r_n$ are linearly independent over $\ZZ$.
\end{definition}

In the remainder of the paper we will simply refer to a higher dimensional arithmetic progression simply as an "arithmetic progression".

\begin{remark}
The rank of a generalized arithmetic progression is simply the largest cardinality of any $\ZZ$-linearly independent subset of $\{r_1,\dots,r_n\}$. It is obvious that the rank of any arithmetic progression is at most its dimension, with equality if and only if the arithmetic progression is li-arithmetic progression.
\end{remark}

\smallskip

Next, let us note here that since our goal is to study the existence of arithmetic progressions of arbitrary length, $\vec{N}=(N_1,\ldots, N_n)$, it is sufficient to restrict to the case $N_1=N_2=\ldots=N_n=:N$. In this case, we will say that the length of the progression is $N \in \NN$.

\smallskip

Next, let us review some standard notations. As usual, for $N \in \NN$ with $N \geq 1$ we denote by $[N]$ the set
\[
[N] :=\{0,1,2,\dots, N\} = \NN \cap [0,N] \,.
\]
$[N]^d$ denotes the Cartesian product of $d$ copies of $[N]$, that is
\[
[N]^d=\{ (k_1,\dots, k_d) : k_j \in [N] \,.  \forall 1 \leq j \leq d \} \,.
\]
We will also need the following definition.

\begin{definition} A \textbf{$d$-dimensional grid of depth $n$} is a set of the form
\[
[k_1,\dots,k_d;l_1,\dots,l_d;n]:=\{ (l_1+m_1k_1,l_2+ m_2k_2,\dots,l_d+m_dk_d) : m_1,\dots, m_d \in [n] \}
\]
for some fixed positive integers $k_1,\dots, k_d$ and fixed $l_1, \dots l_d$.
\end{definition}

\smallskip

Note that a $d$-dimensional grid of depth $n$ is simply an arithmetic progression inside $\ZZ^d$ of dimension $d$ with the ratios
\[
r_j= k_j e_j \,
\]
for some $k_j \in \NN$, where $e_j=(0,0,\dots, 1, 0,\dots,0)$ is the canonical basis.

\smallskip
We now prove the following higher dimensional version of van der Waerden's theorem. For the analogous statement of the  one-dimensional version of this theorem, we refer the reader to \cite{KST, vdW}. Note that are already higher-dimensional generalisations of van der Waerdens's theorem, such as the  Gallai-Witt theorem, see \cite{Maddux} for brief discussion and references therein.

\begin{theorem}[van der Waerden in $\ZZ^d$] \label{theo:vdW} Given any natural numbers $k,r,d$, there exists a number $W(r, k,d)$, such that, no matter how we colour $\ZZ^d$ with $r$ colours, for each $N \geq  W(r,k,d)$, we can find a monochromatic d-dimensional grid\\ $[k_1,\dots,k_d; l_1, \dots, l_d;k] \subseteq [N]^d$ of depth $k$.
\end{theorem}
\begin{proof}
We prove the claim via induction on $d$.

$P(1)$ is the standard van der Waerden theorem \cite{vdW}.

$P(d) \Rightarrow P(d+1)$: Let $r$ and $k$ be given. Let $A$ be the set of all $d$-dimensional grids of depth $k$ which are subsets of $[W(r,k,d)]^d$. Now set
\[
W(r,k,d+1)=W(|A|\cdot r, k, 1)\,.
\]
We show that this choice works. Note first that the van der Waerden theorem is equivalent to the fact that given a set $X$ with $|A| \cdot r$ elements, for any function $v: \NN \to X$ and any $N \geq W(|A|\cdot r, k, 1)$, there exists an element $x \in X$ such that,
\begin{displaymath}
v^{-1}(x) \cap [N]
\end{displaymath}
contains an arithmetic progression of length $k$.

Now, consider any colouring of $\ZZ^{d+1}$ with $r$ colours $c_1,\dots,c_r$. Let $N \geq  W(r,k,d+1)$. Next, for each $1 \leq j \leq N$ consider the colouring of $\ZZ^d \times \{ j \} \subseteq \ZZ^{d+1}$. By $P(d)$, the set $[W(r,k,d)]^d \times \{ j \}$ contains a  monochromatic grid $M_j$ of depth $k$. Let $c(j)$ be the colour of this grid. We can now define a function
\begin{align*}
v&: [N] \to A \times \{c_1,\dots, c_r\}  \\
v(j)&=(M_j, c(j)) \,.
\end{align*}
Then, there exists some $(M,c_l) \in A \times \{c_1,\dots, c_r\}$ such that,
\begin{displaymath}
v^{-1}(M, c_l) \cap [N]
\end{displaymath}
contains an arithmetic progression of length $k$. Let $l_{d+1}, k_{d+1}$ be do that $l_{d+1}+m k_{d+1} \in v^{-1}(M, c_l) \cap [N]$ for all $m \in [k]$. Next, since $M$ is a monochromatic
grid of depth $k$, there exists some $k_1,\dots,k_d;l_1,\dots,l_d$ such that
\[
M=[k_1,\dots,k_d;l_1,\dots,l_d;k] \subseteq [W(r,k,d)]^d \,.
\]
Then, by construction, the grid
\[
[k_1,\dots,k_d, k_{d+1};l_1,\dots,l_d,l_{d+1};k] \subseteq [W(r,k,d)]^d \times [N] \subseteq [N]^{d+1}
\]
is monochromatic of colour $c_l$.  This proves the claim.
\end{proof}

\begin{remark} If we denote by $W(r, k,d)$ the smallest value which satisfies Thm.~\ref{theo:vdW}, then it is obvious that $W(r,k, 1)=W(r,k)$. Moreover, the proof above yields the terrible upper-bound
\[
W(r,k,d+1) \leq W(l, k, 1)\,
\]
where $l=|A|\cdot r$. Note that,
\[
|A|= \left( \sum_{j=0}^{W(n,k,d)} \lfloor \frac{ W(n,k,d)-j }{d} \rfloor \right)^d
\]
which can be seen by observing that, for each $1 \leq i \leq d$ and for every particular choice $1 \leq j \leq W(n,k,d)$ we have
\[
1 \leq l_j \leq  \frac{ W(n,k,d)-j }{d} \,.
\]
\end{remark}

\bigskip

\subsection{Meyer sets and model sets}

In this subsection we review the notion of model sets and Meyer sets in $\RR^d$. For a more detailed review of these definitions and properties we refer the reader to the monograph \cite{TAO} and to \cite{MOO,Moody}.

\medskip

We start by reviewing some of the basic definitions for point sets.

\begin{definition}\label{defps} Let $\Lambda \subseteq \RR^d$ be a point set. We say that $\Lambda$ is
\begin{itemize}
    \item{} \textbf{relatively dense} if there exists some $R>0$ such that for all $x \in \RR^d$ the set $\Lambda \cap B_{R}(x)$ contains at least one point.
      \item{} \textbf{uniformly discrete} if there exists some $r>0$ such that for all $x \in \RR^d$ the set $\Lambda \cap B_{r}(x)$ contains at most a point.
      \item{} \textbf{Delone} if $\Lambda$ is relatively dense and uniformly discrete.
       \item{} \textbf{locally finite} if for all $R>0$ and  $x \in \RR^d$ the set $\Lambda \cap B_{R}(x)$ is finite.
\end{itemize}
\end{definition}

\smallskip

The above definitions are usually stated for arbitrary locally compact Abelian groups (LCAG) $G$ using open and compact sets, respectively. It is easy to see that in the case of $G=\RR^d$ the usual definitions are equivalent to Definition~\ref{defps}.

\smallskip

Next, in the spirit of \cite{MOO} we introduce next the following definition.

\begin{definition}  We say that two Delone sets $\Lambda_1,\Lambda_2$ are \textbf{equivalent by finite translations} if there exists finite sets $F_1,F_2$ such that
\begin{align*}
\Lambda_1 &\subseteq \Lambda_2+F_2 \\
\Lambda_2 &\subseteq \Lambda_1+F_1 \,.
\end{align*}
\end{definition}

\begin{remark}
\begin{itemize}
    \item[(a)] It is easy to see that being equivalent by finite translations is an equivalence relation on the set of Delone subsets of $\RR^d$.
    \item[(b)] By replacing $F_1,F_2$ by $F=F_1 \cup F_2$ one can assume without loss of generality that $F_1=F_2$.
\end{itemize}
\end{remark}

\bigskip
Next we review the notion of cut and project schemes and model sets.

\begin{definition}\label{def:Cps} By a \textbf{cut and project scheme}, or simply CPS, we understand a triple $(\RR^d, H, \cL)$ consisting of $\RR^d$, a LCAG $H$, together with a lattice (i.e. a discrete co-compact subgroup) $\cL \subset \RR^d \times H$, with the following two properties:
\begin{itemize}
    \item{} The restriction $\pi^{\RR^d}|_{\cL}$ of the canonical projection $\pi^{\RR^d}: \RR^d \times H \to \RR^d$ to $\cL$ is a one-to-one function.
    \item{} The image $\pi^H(\cL)$ of the $\cL$ under the canonical projection $\pi^H: \RR^d \times H \to H$ is dense in $H$.
\end{itemize}
In the special case where $H = \RR^m$ then we refer to $(\RR^d, \RR^m, \cL)$ as a \textbf{fully Euclidean} CPS.
\end{definition}
Next we define in the usual way, $L:= \pi^{\RR^d}(\cL)$ and $\star: L \to H$, known as the $\star$-mapping, by
\[
\star= \pi^H \circ ( \pi^{\RR^d}|_{\cL})^{-1} \,.
\]
This allows us rewrite
\[\cL= \{ (x,x^\star) : x \in L \} \,.\]
Note that the range of the $\star$-mapping is\[\star(L)=:L^\star= \pi^H(\cL) \,.\]

\smallskip

We can summarize the CPS in the diagram below.

\smallskip
\begin{center}
\begin{tikzcd}[remember picture]
\RR^d & \arrow[swap]{l}{\pi^{\RR^d}} \RR^d \times H \arrow{r}{\pi^H}& H \\
\pi^{\RR^d}(\cL)&  \arrow[right]{l}{1-1} \cL  \arrow[left]{r}{\text{dense}} & H \\
L \arrow[swap]{rr}{\star} &  & L^\star \\
\end{tikzcd}
\end{center}

\begin{tikzpicture}[overlay,remember picture]
\path (\tikzcdmatrixname-1-1) to node[midway,sloped]{$\supset$}
(\tikzcdmatrixname-2-1);
\path (\tikzcdmatrixname-1-2) to node[midway,sloped]{$\supset$}
(\tikzcdmatrixname-2-2);
\path (\tikzcdmatrixname-1-3) to node[midway,sloped]{$=$}
(\tikzcdmatrixname-2-3);
\path (\tikzcdmatrixname-2-1) to node[midway,sloped]{$=$}
(\tikzcdmatrixname-3-1);
\path (\tikzcdmatrixname-2-3) to node[midway,sloped]{$\supset$}
(\tikzcdmatrixname-3-3);
\end{tikzpicture}
\medskip

\smallskip

We can now review the definition of model sets.

\begin{definition}
Given a CPS $(\RR^d,H,\cL)$ and some subset $W \subseteq H$, we denote by $\oplam(W)$ its pre-image under the $\star$-mapping, that is
\begin{displaymath}
\oplam(W):=\{ x \in L : x^\star \in W \} = \{ x \in \RR^d : \exists y \in W \, \mbox{ such that } (x,y) \in \cL \} \,.
\end{displaymath}

When $W$ is precompact and has non-empty interior, the set $\oplam(W)$ is called a \textbf{model set}. If furthermore $H=\RR^m$ for some $m$, then $\oplam(W)$ is called a \textbf{fully Euclidean model set}.
\end{definition}
We want to emphasize here that the condition $W$ has non-empty interior is essential later in the paper in the proof of Thm.~\ref{thm:ap rank model}.

\medskip
Next, let us recall the following result.
\medskip

\begin{proposition}\label{prop2.6}\cite[Prop.~2.6]{MOO}\cite[Prop.~7.5]{TAO}  Let $(\RR^d, H, \cL)$ be a CPS and $W \subseteq H$.
\begin{itemize}
    \item[(a)] If $W \subseteq H$ is pre-compact, then $\oplam(W)$ is uniformly discrete.
    \item[(b)] If $W \subseteq H$ has non-empty interior, then $\oplam(W)$ is relatively dense.
\end{itemize}
In particular, every model set is a Delone set.
\end{proposition}

\medskip

Next, we review the concept of Meyer set. We start by recalling the following theorem.

\begin{theorem}\cite{LAG1,Meyer,MOO}\label{char mey} Let $\Lambda \subseteq \RR^d$ be relatively dense. Then, the following are equivalent.
\begin{itemize}
    \item[(i)] $\Lambda$ is a subset of a model set.
    \item[(ii)] There exists a fully Euclidean model set $\oplam(W)$ and a finite set $F$ such that
\[
\Lambda \subseteq \oplam(W)+F\,.
\]
\item[(iii)] $\Lambda -\Lambda$ is uniformly discrete.
\item[(iv)] $\Lambda$ is locally finite and there exists a finite set $F_1$ such that
\[
\Lambda - \Lambda \subseteq \Lambda +F_1 \,.
\]
\end{itemize}
\end{theorem}
\begin{proof} The equivalence (i)$\Leftrightarrow$(iii)$ \Leftrightarrow$(iv) can be found in \cite{MOO}, while (ii)$\Leftrightarrow$(iv) is \cite[Thm.~IV]{Meyer} .
\end{proof}

The above theorem gives the concept of a Meyer set. More precisely, we have the following definition.

\begin{definition}
A  relatively dense set $\Lambda \subset \RR^n$ is called \textbf{Meyer set} if it satisfies one (and hence all) equivalent conditions of Thm.~\ref{char mey}.

A relatively dense set $\Lambda$ is called \textbf{fully Euclidean Meyer set} if  there exists some fully Euclidean model set $\oplam(W)$ such that $\Lambda \subseteq \oplam(W).$
\end{definition}

One of the goals of this paper is to characterise fully Euclidean Meyer sets.

\medskip

Next, let us recall the following result from \cite{MOO}.

\begin{lemma} Let $\Lambda \subseteq \RR^d$ be a Meyer set. Then, the group
$\langle \Lambda \rangle$ generated by $\Lambda$ is finitely generated. In particular, $\langle \Lambda \rangle$ is a free $\ZZ$-module of finite rank.
\end{lemma}
\begin{proof}
By \cite[Thm.~9.1]{MOO}, $\langle \Lambda \rangle$ is finitely generated. It is therefore a finitely generated $\ZZ$-module. Moreover, since $\RR$ is torsion free as a $\ZZ$-module, so is $\langle \Lambda \rangle$. Therefore, $\langle \Lambda \rangle$ is a free $\ZZ$-module by \cite[Thm.~12.5]{DF}.
\end{proof}

This allows us introduce the following definition.

\begin{definition} Let $\Lambda \subseteq \RR^d$ be any Meyer set. We define the \textbf{rank} of $\Lambda$ to be
\[
\rank(\Lambda):= \rank_{\ZZ} (\langle \Lambda \rangle)\,.
\]
\end{definition}

\medskip

We complete the section by  showing that each Meyer set is equivalent by finitely many translates with a fully Euclidean model set.

\begin{lemma}\label{equivft} Let $\Lambda$ be a Meyer set, $F$ finite and $\oplam(W)$ a fully Euclidean model set such that
\[\Lambda \subseteq \oplam(W) +F \,.\]
Then $\Lambda$ and $\oplam(W)$ are equivalent by finite translations.
\end{lemma}
\begin{proof}

Since $\oplam(W)$ is a Meyer set and $F$ is finite, $\oplam(W) +F$ is also a Meyer set \cite{Meyer,MOO}. The claim follows now from \cite[Lemma~5.5.1]{NS11}.
\end{proof}

\section{Higher dimensional arithmetic progressions in Meyer sets}\label{S3}

In this section we show that Meyer sets contain arithmetic progressions of arbitrary dimensions and length. The proofs will show that the existence of arithmetic progressions of arbitrary length in Meyer sets is interesting only in the case of li-arithmetic progressions. We will study these in the subsequent sections.

We start by recalling the following well-known theorem.

\begin{lemma}[Chinese Remainder Theorem]\cite[Cor.~7.18]{DF}
If $k_1,\dots, k_n$ are pairwise coprime and $a_1,\dots,a_n$   are integers then there exists $x \in \mathbb{Z}^+$ such that
\[
\left\{
\begin{array}{rl}
x &\equiv a_1 \ \pmod{k_1} \\
x &\equiv a_2 \ \pmod{k_2} \\
&  \hspace{0.95cm} \vdots   \\
x & \equiv a_n \ \pmod{k_n}\,.
\end{array}
\right.
\]
Moreover, any two solutions are congruent modulo $k_1k_2\dots.k_n$.
\end{lemma}

As an immediate consequence we get the following result.

\begin{corollary}\label{cor:higher dim in 1d}
For each $n,N \in \mathbb{N}$ there exists $m_1,\dots,m_n \in \mathbb{N}$ such that $\sum^{n}_{i=1}c_im_i$ are distinct for all $0 \leq c_i \leq N$.
\end{corollary}
\begin{proof}
Let $p_1,\dots,p_n$ be distinct primes such that for all $1 \leq i \leq n$ we have $p_i > N$. By the Chinese Remainder Theorem, there exists $m_1,\dots,m_n$ such that for each $1 \leq i \leq n$ we have
\[
\left\{
\begin{array}{rl}
 m_i &\equiv 1 \ \pmod{p_i} \\
 m_i &\equiv 0 \ \pmod{p_j} \quad \forall j \neq i \,.
\end{array}
\right.
\]
Now if $ \sum^{n}_{i=1}c_im_i =\sum^{n}_{i=1}c'_im_i $
then, for all $1 \leq k \leq n$, we have
\[
\sum^{n}_{i=1}c_im_i \equiv  \sum^{n}_{i=1}c'_im_i \   \pmod{p_k} \,.
\]
and hence $c^{}_k \equiv c'_{k} \pmod{p^{}_{k}}$. Since \[0 \leq c_{k}, c'_{k} \leq N < p_k \,,\]
we get $c_k=c'_k$ for all $1 \leq k \leq n$.\\
\end{proof}
We can now prove the following result.

\begin{proposition}\label{prop1}
Let $n,N \in \mathbb{N}$ and let $\Lambda \in \RR^d$ be a Meyer set. Then there exists some $R >0$ such that $\Lambda \cap B_R(x)$ contains a non-trivial and proper $n$-dimensional arithmetic progression of length $N$ for all $x \in \mathbb{R}^d$.
\end{proposition}

\begin{proof}

Let $n, N \in \mathbb{N}$ be given. We show that, for some $R >0$, there exists \[
s, r_1, r_2, \dots , r_n \in \Lambda
\]
such that
\[
s + \sum^{n}_{i=1}c_ir_i \in \Lambda \cap B_R(x) \qquad \forall  0\leq c_i\leq N \,,
\]
and that the elements $s + \sum^{n}_{i=1}c_ir_i $ are distinct. Let $m_1,\dots, m_n$ be as in Cor.~\ref{cor:higher dim in 1d} and set $N'= Nm_1+\dots+Nm_n$. By \cite[Lemma 4.3]{KST} there exists an $R>0$ and some non-zero $s, r \in \Lambda$ such that
\[
s, s+r, \dots , s+N'r \in \Lambda \cap B_R(x)\,.
\]
Now define $r_j:= m_jr$. It follows that, for $0 \leq \sum^{n}_{i=1}c_im_i \leq N'$,
\[
 s + \sum^{n}_{i=1}c_ir_i = s+ \sum^{n}_{i=1}c_im_ir  \in \Lambda \cap B_R(x)\,.
\]
Moreover, by Cor.~\ref{cor:higher dim in 1d} the progression is proper.
\end{proof}

Note here in passing that, by construction the arithmetic progression in Prop.~\ref{prop1} has rank 1.

\smallskip

It becomes natural to ask if one can construct arithmetic progressions of higher rank. It is easy to see that one can focus on li-arithmetic progressions. Indeed, exactly as in Prop.~\ref{prop1} one can prove the following result.

\begin{lemma}\label{lem1} Let $\Lambda \subseteq \RR^d$ be any Meyer set and $k \in \NN$.
Then, $\Lambda$ contains li-arithmetic progressions of rank $k$ and arbitrary length if and only if for each $d \geq k$, $\Lambda$ contains arithmetic progressions of rank $k$, dimension $d$, and arbitrary length.
\end{lemma}
Since the proof is similar to the one of Prop.~\ref{prop1}, we skip it.

\bigskip

\section{Higher dimensional arithmetic progressions with linearly independent ratios }\label{S4}

In this section we discuss the existence of li-arithmetic progressions in fully Euclidean model sets. We show that the maximal rank of any such progression is the rank of the lattice in the CPS, and that for this rank, we can find li-arithmetic progressions of arbitrary length.

\smallskip

We start by proving the following result (compare \cite[Prop.~2.6]{MOO}).

\begin{lemma}\label{lem:model set generates L}  Let $(\RR^d, \RR^m, \cL)$ be a fully Euclidean CPS, and let $W \subseteq \RR^m$ be any set with non-empty interior. Then $\oplam(W)$ generates $L=\pi_{\RR^d}(\cL)$. In particular,
\[
\rank(\oplam(W)) =d+m \,.
\]
In particular, there exists vectors $r_1,\dots,r_{d+m} \in \oplam(W)$ which are linearly independent over $\ZZ$.
\end{lemma}
\begin{proof}

It suffices to show that $\oplam(W)-\oplam(W)=\oplam(W-W)$ generates $L$. Note that $0$ is an interior point in $W-W$ and hence we can find some $r>0$ such that $B_r(0) \subseteq W-W$. \\

Now, let $x \in L$ be arbitrary. Pick some $n$ such that $d(x,0) < nr$. Then $\frac{x}{n} \in B_r(0)$. Note here that $\frac{x}{n} \in B_r(0) \cap B_{\frac{r}{n}}(\frac{x}{n})$. Since this set is open, by the density of $\pi_{\RR^m}(\cL)$ in $\RR^m$, there exists some $(z,z^\star) \in \cL$ such that $z^\star \in B_r(0) \cap B_{\frac{r}{n}}(\frac{x}{n})$. Then $z \in \oplam(W-W)$ and $d(x^\star, nz^\star)<r$, thus $x^\star-nz^\star \in B_r(0) \subseteq W-W$. Therefore,
\begin{align*}
z & \in \oplam(W)- \oplam(W) \\
x-nz & \in \oplam(W)-\oplam(W)\,.
\end{align*}
This gives  $x \in \langle \oplam(W) \rangle$. The last claim follows now from Lemma~\ref{lem:12}.
\end{proof}

\smallskip
Next, we prove the following generalisation of \cite[Prop~4.2]{KST}.

\begin{proposition} \label{prop:model set AP} Let $(\RR^d,H, \cL)$ be any CPS and let $W \subseteq H$ be any set with non-empty interior. Then, for all $M \in \NN$, we can find open sets $U_M, V_M \subset H$ such that $0 \in V_M$ and
\begin{displaymath}
\oplam(U_M)+\underbrace{\oplam(V_M)+\oplam(V_M)+\dots+\oplam(V_M)}_{M -\mbox{\rm times }} \subseteq  \oplam(W) \,.
\end{displaymath}
\end{proposition}
\begin{proof}
As $W$ has non-empty interior, there exists non-empty open sets $U_M \subseteq H$ and $0 \in V_M \subseteq H$ such that
\[
U_M + \underbrace{V_M+V_M+\dots+V_M}_{M \mbox{-times }}\subseteq W\,.
\]
Then,
\[
\oplam(U_M)+\underbrace{\oplam(V_M)+\oplam(V_M)+\dots+\oplam(V_M)}_{M\mbox{-times }}
 \subseteq \oplam(U_M + \underbrace{V_M+V_M+\dots+V_M}_{M\mbox{-times}}) \subseteq \oplam(W)\,.
\]
This completes the proof.
\end{proof}

By combining Prop.~\ref{prop:model set AP} with Cor.~\ref{cor1} we get the following result.

\begin{theorem}\label{thm:ap rank model}
\label{theo:d+m AP}
Let $\oplam(W)$ be a model set in a fully Euclidean CPS $(\RR^d, \RR^m, \cL)$. Then,
\begin{itemize}
\item[(a)] Any arithmetic progression of length $N$ in $\oplam(W)$ has rank at most $d+m$.
\item[(b)] For each $N$ there exists some $R>0$ such that the set $\oplam(W) \cap B_{R}(y)$ contains an li-arithmetic progression of length $N$ and rank $d+m$ for all $y \in \RR^d$.
\end{itemize}
\end{theorem}
\begin{proof}
\textbf{(a)}
For any arithmetic progression of rank $k$ in $\oplam(W)$ the set $\{r_1^{},\dots,r_k^{} \} $, with $r_i^{}$ the ratios of the progression, is $\mathbb{Z}-$linearly independent in $L$. Therefore, by Cor.~\ref{cor1} we have $k \leq m+d$. This proves (a). \\

\textbf{(b)} Let $N$ be given. We show that there exists some $R>0$, such that, for all $y \in \RR^d$ there exists some $s \in \oplam(W)$, and $\mathbb{Z}-$linearly independent $r_1^{},\dots,r_{m+d}^{}$ with $s + \sum_{j=1}^{m+d}c_jr_j \in \oplam(W) \cap B_R(y)$ for all $0 \leq c_j^{} \leq N.$ \\

Set $M:= N\cdot(m+d) \,.$ By Prop.~\ref{prop:model set AP} there exists open sets $U_M, V_M \subseteq \RR^m$ such that $0 \in V_M$ and
\[
\oplam(U_M)+\underbrace{\oplam(V_M)+\oplam(V_M)+\dots+\oplam(V_M)}_{M-\text{times}} \subseteq  \oplam(W)\,. \]
As $U_M^{}$ has non-empty interior, by Prop.~\ref{prop2.6} there exists $R' > 0$ such that $\oplam(U_M) + B_{R'}(0) = \mathbb{R}^d$. Next, by Lemma~\ref{lem:12} there exists $\ZZ$-linearly independent vectors $r_1^{},\dots,r_{m+d}^{} \in \oplam(V_M)$.
Define
\[
R:= N\cdot(\|r_1^{}\|+\dots+\|r_{m+d}^{}\|)+R'\,.
\]

Let $y \in \mathbb{R}^d$ be arbitrary. By the definition of $R'$, there exists some $s \in \oplam(U_M^{}) \cap B_{R'}(y).$ Then, for all $ 0 \leq c_i \leq N$, as $0 \in \oplam(V_M)$, we have
\begin{align*}
s+\sum_{i=1}^{m+d} c_i^{} r_i^{} &=s+ \underbrace{r_1^{}+\dots+r_1^{}}_{c_1^{}\text{times}}+\dots+\underbrace{r_{m+d}^{}+\dots+r_{m+d}^{}}_{c_{m+d}^{}\text{times}}\\
&= s+\underbrace{r_1^{}+\dots+r_1^{}}_{c_1^{}\text{times}}  +\dots + \underbrace{r_{m+d}^{}+\dots+r_{m+d}^{}}_{c_{m+d}^{}\text{times}} \,\,\, + \underbrace{0+\dots+0}_{M-\sum_{i=1}^{m+d}c_i^{}\text{times}}\\
& \in \oplam(U_M)+\underbrace{\oplam(V_M)+\oplam(V_M)+\dots+\oplam(V_M)}_{ c_1 \ \text{times}}+\underbrace{\oplam(V_M)+\oplam(V_M)+\dots+\oplam(V_M)}_{ c_2 \ \text{times}} \\
&+  \dots +\underbrace{\oplam(V_M)+\oplam(V_M)+\dots+\oplam(V_M)}_{ c_{m+d} \ \text{times}} +\underbrace{\oplam(V_M)+\oplam(V_M)+\dots+\oplam(V_M)}_{M-\sum_{i=1}^{m+d}c_i^{} \ \text{times}}   \\
& = \oplam(U_M)+\underbrace{\oplam(V_M)+\oplam(V_M)+\dots+\oplam(V_M)}_{ M \ \text{times}} \subseteq \oplam(W)  \,.
\end{align*}

Moreover, for all $ 0 \leq c_i \leq N$, we have
\begin{align*}
d(s+\sum_{i=1}^{m+d}c_i^{}r_i,y)& = \| y- \left(s+\sum_{i=1}^{m+d}c_i^{}r_i \right) \| \leq \| y- s \| + \| \sum_{i=1}^{m+d}c_i^{}r_i  \| \\
&\leq R' +\sum_{i=1}^{m+d} \| c_i^{}r_i  \| =  R' +\sum_{i=1}^{m+d}c_i^{} \| r_i  \| \leq R' +\sum_{i=1}^{m+d}N \| r_i  \|= R \,.
\end{align*}

This implies that, for each $0 \leq c_i^{} \leq N$,
\[
s+ \sum_{i=1}^{m+d} c_i^{}r_i^{} \in \oplam(W) \cap B_R(y)\,.
\]
\end{proof}

\smallskip

Thm.~\ref{thm:ap rank model} suggests the following definition.

\begin{definition}  Let $\Lambda \subseteq \RR^d$ be a Meyer set.

The \textbf{ap-rank} of $\Lambda$, denoted by $\aprank(\Lambda)$, is the largest $k \in \NN$ with the property that for all $N \in \NN$, there exists an li-arithmetic progression of length $N$ and rank $k$ in $\Lambda$.
\end{definition}

Note here in passing that by Lemma~\ref{lem1}, the ap-rank of $\Lambda$ is the largest positive integer $k$ such that for all $N \in \NN$ the set $\Lambda$ contains an arithmetic progression of length $N$ and rank $k$.

\medskip

The next result tells us that the definition of ap-rank makes sense.

\begin{lemma}\label{ll1} Let $\Lambda \subseteq \RR^d$ be any Meyer set. Then
\[
1 \leq \aprank(\Lambda)  \leq \rank(\Lambda) \,.
\]
\end{lemma}
\begin{proof}
The lower bound follows immediately from \cite{KST}. Note here that any non-trivial 1-dimensional arithmetic progression in $\RR^d$ has rank 1.

Next, consider any arithmetic progression of length $N \geq 2$ and rank $k$ in $\Lambda$ and let $r_1,\dots,r_k$ be the linearly independent ratios. Then
$r_1,\dots,r_k$ are linearly independent vectors in $\langle \Lambda \rangle$. The claim follows from Cor.~\ref{cor1}.
\end{proof}

\begin{remark}
\begin{itemize}
\item [(a)] Thm.~\ref{theo:d+m AP} says that for a fully Euclidean model set $\oplam(W)$ in the CPS $(\RR^d,\RR^m,\cL)$ we have
\[
\rank(\oplam(W)) = \aprank(\oplam(W)) =d+m \,.
\]
\item [(b)] Since the ap-rank does not change under translates, but the rank changes for translates outside the $\ZZ$-module generated by the set, it is easy to construct examples of translates of fully Euclidean model sets such that
\[
\rank(t+\oplam(W)) = \aprank(t+\oplam(W)) +1\,.
\]
In Example~\ref{ex1}, for each $n \in \NN$,  we will construct a Meyer set $\Lambda$ such that
\[
\rank(\Lambda) = \aprank(\Lambda) +n \,.
\]
\item [(c)] If $\Lambda \subseteq \RR^d$ is a Meyer set, we will show in Cor.~\ref{cor3} that
\[
\aprank(\Lambda) \geq d \,.
\]
\end{itemize}
\end{remark}

\medskip

We complete the section by providing a colouring version of Thm.~\ref{theo:d+m AP}.

\begin{theorem} \label{theo:mono model}
 Let $\oplam(W)$ be a model set in a fully Euclidean CPS $(\RR^d, \RR^m, \cL)$. Then, for each $r,k$ there exists some $R$ such that, no matter how we colour $\oplam(W)$ with $r$ colours the set $\oplam(W) \cap B_{R}(y)$ contains a monochromatic li-arithmetic progression of length $k$ and rank $d+m$ for all $y \in \RR^d$.
\end{theorem}
\begin{proof}
Let $N$ be such that van der Waerden's Theorem (Thm.~\ref{theo:vdW}) holds for $r,k$ applied to $[N]^{d+m}.$ By Thm.~\ref{theo:d+m AP} there exists some $R>0$ such that, for all $y \in \mathbb{R}$, the set
\[
\oplam(W) \cap B_R(y)\,,
\]
contains a non-trivial li-arithmetic progression length $N$. We show that this $R$ works.

Arbitrarily colour \oplam(W) with $r$ colours and let $y \in \mathbb{R}^d$ be given. By Thm.~\ref{theo:d+m AP} there exists $s, r_1^{},\dots,r_{m+d}^{}$ such that $r_1^{},\dots,r_{m+d}^{}$ are $\ZZ$-linearly independent, and for all $0 \leq c_i \leq N$,
\[
 s+\sum_{i=1}^{m+d}c_ir_i^{} \in \oplam(W) \cap B_R(y) \,.
\]
Colour the set $[N]^{m+d}$ by colouring $(c_1,\dots,c_{m+d})$ with the colour of $s+\sum_{i=1}^{m+d}c_ir_i^{}$. By the choice of $N$, there exists a monochromatic grid
\[
[k_1,\dots,k_{d+m};l_1,\dots,l_{d+m};k]\in [N]^{d+m}
\]
of length $k$ and dimension $d+m$. Then, for all $1 \leq m_j \leq k$, the set
\[
s+\sum_{j=1}^{m+d}(l_j+m_jk_j)r_j \in \oplam(W) \cap B_R(y)
\]
is monochromatic. Now set $s':=s+\sum_{j=1}^{m+d} l_jr_j$ and $r'_j:=k_jr_j$. Then, for all
$(m_1,\dots, m_{m+d}) \in [k]^{m+d}$,
\[
s'+\sum_{j=1}^{m+d} m_j r'_j\in \oplam(W) \cap B_R(y)
\]
is a monochromatic li-arithmetic progression of length $k$ and rank $m+d.$

\end{proof}

\section{Ap-rank of Meyer sets}\label{S5}

In this section we calculate the ap-rank of a Meyer set $\Lambda$. We know that $\Lambda$ is equivalent by finite translates to a fully Euclidean model set $\oplam(W)$, and we use this to show that
\[
\aprank(\Lambda)=\aprank(\oplam(W))=\rank(\oplam(W)) \,.
\]

\medskip

We start by proving the following result.

\begin{lemma}\label{lema: equiv same ap} Let $\Lambda, \Gamma \subseteq \RR^d$ be Meyer sets which are equivalent by finite translates. Then
\[
\aprank(\Lambda)= \aprank(\Gamma)
\]
\end{lemma}
\begin{proof}
By symmetry it suffices to show that
\[
\aprank(\Lambda) \leq  \aprank(\Gamma)\,.
\]
Let $F = \{f_1, \dots, f_r\}$ be such that
\[
\Lambda \subseteq \Gamma+F
\]
and let $k=\aprank(\Lambda)$. We show that $\Gamma$ contains arithmetic progressions of length $N$ and rank $k$.

First, colour $F$ with $r$ colours, such that no two points of $F$ have the same colour. Next, let $N \in \NN$ be arbitrary and let $N'=W(|F|,N,d)$ be given by Thm.~\ref{theo:vdW}. Since $k=\aprank(\Lambda)$, by definition, there exists an li-arithmetic progression
\[
\{ s+ c_1r_2+\dots.+c_kr_k : 0 \leq c_j \leq N' \}
\]
of length $N'$ and rank $k$ in $\Lambda$. Now, for each $(c_1,\dots,c_k) \in [0,N']^k$ there exists some and $f_j \in F$ so that
\[
 x:= s+ c_1r_2+\dots.+c_kr_k \in \bigcup^{r}_{j=1} \Gamma +f_j\,.
\]
Pick the smallest $j$ such that $x = y+f_j$ for some $y \in \Gamma.$ Colour $(c_1,\dots,c_k)$ with the colour of this $f_j \in F$. Then, by Thm.~\ref{theo:vdW}, there exists a monochromatic grid $[c'_1d_1,\dots,c'_kd_k;k] \subseteq [N']^k$ of depth $N$. Define
\begin{align*}
s'&=s-f_j \\
r_j'&=d_jr_j
\end{align*}
Then, for all $0 \leq c'_j \leq N$
\[
 x = s'+ c'_1r_1'+\dots.+c'_kr'_k \in \Gamma
\]
is an li-arithmetic progression of length $N$ and rank $k$. Since $N \in \NN$ is arbitrary, we get  $k \leq \aprank(\Gamma)$.

\end{proof}

Now, by combining all results so far we get the following Theorem, which is the first main result in the paper.

\begin{theorem}\label{theo:meyer AP}
Let $\Lambda \subseteq \RR^d$ be a Meyer set, and let $\oplam(W)$ be any fully Euclidean model set in $(\RR^d, \RR^m, \cL)$
 and $F \subseteq \RR^d$ be finite such that
\[
\Lambda \subseteq \oplam(W) +F \,.
\]
Then
\[
\aprank(\Lambda)=d+m= \aprank(\oplam(W))=\rank(\oplam(W)) \,.
\]

Moreover, for each $N$ there exists some $R>0$ such that the set $\Lambda \cap B_{R}(y)$ contains an li-arithmetic progression of length $N$ and rank $d+m$ for all $y \in \RR^d$.
\end{theorem}
\begin{proof}

By Lemma~\ref{equivft}, $\Lambda$ and $\oplam(W)$ are equivalent by finite translates. Therefore, by Lemma~\ref{lema: equiv same ap} we have
\[
\aprank(\Lambda)= \aprank(\oplam(W)) \,.
\]
Moreover, by Thm.~\ref{thm:ap rank model} we have
\[
\rank(\oplam(W))=\aprank(\oplam(W))=d+m \,.
\]
Next, since $\Lambda$ and $\oplam(W)$ are equivalent by finite translates, there exists some finite set $F= \{t_1^{},\dots,t_r^{} \}$ such that
\[
\oplam(W) \subseteq \Lambda +F.
\]
Colour $F$ with $r$ colours so that each point of $F$ has a different colour. Let $R'$ be the constant given by Thm.~\ref{theo:mono model} for  $\oplam(W)$ with $r$ colours and length $N$. Define
\[
R := \max\{R' + ||t_j^{}|| : 1 \leq j \leq r \}\,.
\]
Colour $\oplam(W)$ the following way: for each $x \in \oplam(W)$, there exists some minimal $j$ such that $x \in t_j^{} + \Lambda$. Colour each $x$ by the colour of $t_j^{} $ for this minimal $j$. This gives an $r$-colouring of $\oplam(W)$. Note here that any choice of $t_j$ works, but one needs to make a choice in case some $x \in \oplam(W)$ belongs to $t_j^{} + \Lambda$ for more than one $j$.\\

Now let $y \in \mathbb{R}^d$ be arbitrary. By Thm.~\ref{theo:mono model} there exists a monochromatic li-arithmetic progression $s+\sum_{i=1}^{m+d}c_i^{}r_i^{}\in \oplam(W) \cap B_{R'}^{}(y)$ of rank $m+d$, for all $0 \leq c_i^{} \leq N$.

Since the arithmetic progression is monochromatic, there exists some $j$ such that, for all $0 \leq c_i \leq N$, we have
\[
s+\sum_{i=1}^{m+d}c_i^{}r_i^{} \in t_j^{} + \Lambda\,.
\]
Thus, for all $0 \leq c_i^{} \leq N$, we have
\[
s-t_j^{}+\sum_{i=1}^{m+d}c_i^{}r_i \in \Lambda
\]
which is an li-arithmetic progression of length $N$ and rank $m+d$. Moreover, for each $0 \leq c_i^{}, \leq N$
\[
d(s-t_j^{} +\sum_{i=1}^{m+d}c_i^{}r_i, y) =||s+\sum_{i=1}^{m+d}c_i^{}r_i - t_j^{} -y|| \leq ||s+\sum_{i=1}^{m+d}c_i^{}r_i -y|| + ||t_j^{}|| \leq R \,.
\]
This proves the last claim.

\end{proof}

We start by listing some consequences of this result.

A first immediate consequence we get that if $\Lambda$ is a Meyer set, $\oplam(W)$ a fully Euclidean model set in $(\RR^d, \RR^m, \cL)$ and $F$ is a finite set such that
\[
\Lambda \subseteq \oplam(W) +F
\]
then $m$ is invariant for $\Lambda$.

\begin{corollary} Let $\Lambda \subseteq \RR^d$ be any Meyer set. If $(\RR^d, \RR^m, \cL)$ and  $(\RR^d, \RR^n, \cL)$ are two different CPS, $W \subseteq \RR^m, W' \subseteq \RR^n$ are precompact windows with non-empty interior and $F,F'$ are finite sets such that
\begin{align*}
  \Lambda &\subseteq \oplam(W)+F \\
  \Lambda &\subseteq \oplam(W')+F' \,,
\end{align*}
then $m=n$.
\end{corollary}
\begin{proof}
By Thm.~\ref{theo:meyer AP} we have
\begin{align*}
  \aprank(\Lambda) &=d+m \\
  \aprank(\Lambda) &=d+n \,.
\end{align*}
\end{proof}
 \smallskip

As an immediate consequence, we get the following improvement on the lower bound from Lemma~\ref{ll1}.

\begin{corollary}\label{cor3} Let $\Lambda \subseteq \RR^d$ be any Meyer set. Then
\[
\aprank(\Lambda)\geq d \,.
\]
\end{corollary}

Next, we show that in general there is no upperbound for $\aprank(\Lambda)$ in terms of $\rank(\Lambda)$.

\begin{example}\label{ex1} Let $\oplam(W)$ be any fully Euclidean model set and let $r_1,\dots,r_k$ be so that
\[
\langle \oplam(W) \rangle = \oplus_{j=1}^k \ZZ r_k \,.
\]
Let $s_1,\dots,s_n$ be so that $r_1,\dots,r_k, s_1,\dots,s_n$ are linearly independent over $\ZZ$ and let $F= \{s_1,\dots,s_n\}$. Then
\[
\Lambda = \oplam(W)+ F
\]
is a Meyer set and
\begin{align*}
\aprank(\Lambda)&=k \\
\rank(\Lambda)&=k+n\,.
\end{align*}
\end{example}

\medskip

\begin{remark} If $\oplam(W)$ is a fully Euclidean model set with $k=\aprank(\Lambda)$ then, by Thm.~\ref{thm:ap rank model} \textbf{every} li-arithmetic progression in $\oplam(W)$ has rank at most $k$. The same is not true in Meyer sets.

Indeed let $k \leq d$ and $N$ be any positive integers. Let $A$ be any li-arithmetic progression of rank $d$ and length $N$ and let $\oplam(W)$ be any fully Euclidean model set of rank $k$ such that $0 \in \oplam(W)$. Then
\[
\Lambda := \oplam(W)+A
\]
is a Meyer set, with $\aprank(\Lambda)=k$, which contains the li-arithmetic progression $A$ of rank $d$ and length $N$.

Note that $\aprank(\Lambda)=k$ means that for all $d >k$, if $\Lambda$ contains li-arithmetic progressions of rank $d$, then they are bounded in length. We will see later in Cor.~\ref{cor2} that for fully Euclidean Meyer sets, the rank of every \textbf{every} li-arithmetic progression is also bounded by the ap-rank.
\end{remark}

\bigskip

We complete the section by extending, as usual, Thm.~\ref{theo:meyer AP} to colourings of $\Lambda$.

\begin{theorem}
Let $\Lambda \subset \RR^d$ be Meyer set, and let $k=\aprank(\Lambda)$. Then, for each $r,N$ there exists some $R$ such that, no matter how we colour $\Lambda$ with $r$ colours, the set $\Lambda \cap B_{R}(y)$ contains a monochromatic li-arithmetic progression of length $N$ and rank $k$ for all $y \in \RR^d$.
\end{theorem}
\begin{proof}
Pick $N'$ such that van der Waerden's Theorem holds for $r,k$ applied to $[N']^{d+m}$. By Thm.~\ref{theo:meyer AP} there exists $R>0$ such that for all $y \in \mathbb{R}^d$ the set $\Lambda \cap B_R^{}(y)$ contains an arithmetic progression of length $N$ and dimension $k$. The rest of the proof is identical to that of Thm.~\ref{theo:mono model}.
\end{proof}

\section{A characterisation of fully Euclidean Meyer sets}

We complete this paper by characterizing fully Euclidean Meyer sets. To our knowledge, this is the first result in this direction.

\begin{theorem}\label{fully euc meyer} Let $\Lambda \subseteq \RR^d$ be a Meyer set. Then $\Lambda$ is a fully Euclidean Meyer set if and only if
\[
\aprank(\Lambda)=\rank(\Lambda)\,. \\
\]
\end{theorem}
\begin{proof}
$\Longrightarrow$

Let $\oplam(W)$ be a fully Euclidean model set such that
\[\Lambda \subseteq \oplam(W)\,.\]
Then, we have $\rank(\Lambda) \leq \rank(\oplam(W))$. Therefore, by Lemma~\ref{ll1} and Thm.~\ref{theo:meyer AP},
\[
\aprank(\Lambda) \leq \rank(\Lambda) \leq \rank(\oplam(W)) = \aprank(\Lambda) \,.
\]
This gives
\[
\aprank(\Lambda) = \rank(\Lambda) \,.
\]

\medskip

$\Longleftarrow$

Let $k:= \aprank(\Lambda)=\rank(\Lambda) \,.$ By Thm.~\ref{char mey} there exists a CPS $(\RR^d, \RR^m, \cL)$, a window $W \subseteq \RR^m$ and a finite set $F=\{ f_1,\dots, f_l \}$ such that
\begin{equation}\label{eq111}
\Lambda \subseteq \oplam(W)+F \,.
\end{equation}
Moreover, without loss of generality we can assume that no proper subset $F'$ of $F$ satisfies \eqref{eq111}. We start by showing that there exists some $n$ so that $n \Lambda \subseteq L=\pi_{\RR^d}(\cL)$.

\smallskip

First, note that by Thm.~\ref{theo:meyer AP} we have
\[
\rank(\oplam(W))=\rank(\oplam(W-W))= k\,.
\]
By \eqref{eq111}, we can partition $\Lambda$ as
\begin{align*}
\Lambda &=\bigcupdot_{j=1}^l \Lambda_j \\
\Lambda_j &\subseteq \oplam(W) +f_j \,.
\end{align*}
Note that for all $1 \leq j \leq l$ we have
\[
\Lambda_j - \Lambda_j \subseteq \oplam(W)- \oplam(W)
\]
and hence
\[
\Gamma:= \bigcup_{j=1}^l (\Lambda_j -\Lambda_j) \subseteq \oplam(W)-\oplam(W) \,.
\]
We claim that $\Gamma:= \bigcup_{j=1}^l (\Lambda_j -\Lambda_j)$ is relatively dense. Indeed, set
\[
J:= \{ j : 1 \leq j \leq l, \Lambda_j \neq \emptyset \} \,.
\]
Then,
\[
\Gamma= \bigcup_{j \in J} \left( \Lambda_j - \Lambda_j \right) \,.
\]

Now, for each $j \in J$ fix some $x_j \in  \Lambda_j$, which exists by the definition of $J$. Let
\[
F'= \{ x_j : j \in J \}\,.
\]
Then, we have
\begin{align*}
\Lambda &= \bigcup_{j \in J} \Lambda_j=  \bigcup_{j \in J} (\Lambda_j -x_j+x_j) \subseteq \bigcup_{j \in J} (\Lambda_j -x_j+F')\\
&=(\bigcup_{j \in J} (\Lambda_j -x_j))+F' \subseteq (\bigcup_{j \in J} (\Lambda_j -\Lambda_j))+F'= \Gamma+F' \,.
\end{align*}
Since $\Lambda$ is relatively dense, it follows immediately that $\Gamma$ is relatively dense. In particular, $\Gamma$ is a Meyer set. Now, by Thm.~\ref{theo:meyer AP} we have
\[
\aprank(\Gamma)=\rank(\oplam(W-W))=k \,.
\]
Since $\Gamma \subseteq  \oplam(W-W)$ we have
\[
k=\aprank(\Gamma)\leq \rank(\Gamma) \leq \rank(\oplam(W-W))= k \,,
\]
and hence
$\rank(\Gamma)= k \,.$
\smallskip

Next, let $L_1 = \langle \Lambda \rangle$ and $L_2 = \langle \Gamma \rangle$ by the $\ZZ$-modules generated by $\Lambda$ and $\Gamma$, respectively.
Since for each $1 \leq j \leq l$ we have $\Lambda_j \subseteq \Lambda$ and hence $\Lambda_j -\Lambda_j \subseteq \Lambda-\Lambda \subseteq L_1$, we get that $L_2$ is a $\ZZ$-submodule of $L_1$.

\smallskip

Now, recall that by Lemma~\ref{lem:model set generates L} we have
\[
L= \langle \oplam(W) \rangle = \langle \oplam(W-W) \rangle \,.
\]
Since $\Gamma \subseteq \oplam(W-W)$ we have $\Gamma \subseteq L$ and hence $L_2$ is a submodule of $L$. Moreover, by the above
\[
\rank(L_1)=\rank(L_2)=k
\]
by Lemma~\ref{nM N} there exists some positive integer $n$ such that $nL_1 \subseteq L_2$. Therefore
\[
n \Lambda \subseteq n L_1 \subseteq L_2 \subseteq L \,,
\]
as claimed. Next, let $v_1,\dots, v_{d+m}$ be the vectors so that
\[
\cL= \ZZ v_1 \oplus \ldots \oplus \ZZ v_{m+d} \,.
\]
Now, by enlarging the lattice $\cL$ we can make sure that $\Lambda$ is inside the projection of the lattice on $\RR^d$. Indeed, for each $1 \leq  j \leq m+d$ let
\[
w_j= \frac{1}{n} v_j
\]
and set
\[
\cL':=  \ZZ w_1 \oplus \ldots \oplus \ZZ w_{m+d}  \,.
\]
Then, it is obvious that $(\RR^d, \RR^m, \cL')$ is a CPS and that
\[
L'=\pi_{\RR^d}(\cL')= \frac{1}{m} \pi_{\RR^d}(\cL)= \frac{1}{n} L \,.
\]
Moreover, by construction we have $\cL \subseteq \cL'$. This gives
\[
\Lambda \subseteq \frac{1}{n} L=L' \,.
\]
We complete the proof by showing that $F \subseteq L'$. The result will follow from this.

Let $f \in F$ be arbitrary. By the minimality of $F$ we have
\[
\Lambda \nsubseteq \oplam(W)+(F \backslash \{ f\}) \,.
\]
Therefore, there exists some $x \in \Lambda$ such that $x \notin \oplam(W)+(F \backslash \{ f\})$. However, as $x \in \Lambda \subseteq \oplam(W)+ F$, we get that
\[
x \in \oplam(W) +f  \,.
\]
Thus, there exists some $y \in \oplam(W)$ such that $x=y+f$. It follows that
\[
f=x-y \in \Lambda - \oplam(W) \subseteq L' -L \subseteq L'-L'=L' \,,
\]
as claimed. Therefore, for all $1 \leq j \leq l$, there exists some $g_j \in \RR^m$ such that $(f_j,g_j) \in \cL'$. Define
\[
W'= \bigcup_{j=1}^l g_j +W \subseteq \RR^m \,.
\]
Then, $W'$ is pre-compact and has non-empty interior.

We show that
\[
\Lambda \subseteq \{ x \in L' : \exists y \in W'  \mbox{ such that } (x,y) \in \cL' \}=: \oplam'(W') \,,
\]
which, as $\oplam'(W')$ is a fully Euclidean model set in the CPS $(\RR^d, \RR^m, \cL')$, completes the proof.

Let $x \in \Lambda$ be arbitrary. Then, as $\Lambda \subseteq \oplam(W)+F$, there exists some $y \in \oplam(W)$ and $1 \leq j \leq l$ such that
\[
x=y+f_j \,.
\]
Since $y \in \oplam(W)$, there exists some $z \in W$ so that $(y,z) \in L \subseteq L'$. Then, we have
\begin{align*}
  (x,z+g_j) &=(y,z)+(f_j,g_j) \in \cL+\cL' \subseteq \cL'+\cL'=\cL' \\
  z+g_j &\in g_j+W \subseteq W' \,,
\end{align*}
and hence $x \in \oplam'(W')$.

\end{proof}

\begin{remark} Thm.~\ref{fully euc meyer} can be equivalently stated as follows.

Let $\Lambda \subseteq \RR^d$ be a Meyer set. Then, $\Lambda$ is a fully Euclidean Meyer set if and only if for $k = \rank(\Lambda)$, for each $N \in \NN$ there exist some $s, r_1,\dots,r_k \in \RR^d$ such that $r_1,\dots,r_k$ are linearly independent over $\ZZ$ and
for all $0 \leq C_j \leq N$ we have
\[
s+C_1r_1+\dots+C_kr_k \in \Lambda \,.
\]
\end{remark}

\smallskip

Note here that if $A$ is a li-arithmetic progression of length $N \geq 1$ in a fully Euclidean Meyer set $\oplam(W)$, with ratios $r_1,r_2,\dots,r_k$, then we have $r_1,r_2,\dots,r_k \in \langle \Lambda \rangle$ which gives $k \leq \rank(\oplam(W))$. Therefore, we get the following.

\begin{corollary}\label{cor2} Let $\Lambda \subseteq \RR^d$ be a fully Euclidean Meyer set with $\rank(\Lambda)=k$. Then,
\begin{itemize}
    \item[(a)] Every li-arithmetic progression in $\Lambda$ of length $N \geq 1$ has rank at most $k$.
    \item[(b)] For each $N \in \NN$ there exists a li-arithmetic progression in $\Lambda$ of rank $k$ and length $N$.
\end{itemize}

\end{corollary}

We complete the paper by giving an explicit example of a Meyer set $\Lambda$ which is not fully Euclidean, and we explicitly construct a CPS which produces it as a non-fully Euclidean Meyer set. In fact, $\Lambda$ will be a regular model in this CPS.

\begin{example}
Let $\textbf{Fib}$ denote the well-known Fibonacci model set with its corresponding CPS $(\RR, \RR, \cL)$, where $\cL= \ZZ(1,1) \oplus \ZZ(\tau, \tau')$. We refer the reader to \cite{TAO} for a full description. Note that $\textbf{Fib}$ is a fully Euclidean regular model set within a 2-dimensional CPS. 

Now, take $\Lambda = \pi + \textbf{Fib}$; it follows that $\Lambda$ is still relatively dense and thus a Meyer set. As the ap-rank is invariant under translates, we have
\[
\aprank(\Lambda) = \aprank(\textbf{Fib}) =2 \,.
\]

Now since $\tau$ is an algebraic integer and $\pi$ is transcendental, $1, \tau, \pi$ are linearly independent over $\ZZ$. It is easy to see that 
$$1, \tau, \pi \subseteq \langle \pi+ \textbf{Fib} \rangle  \subseteq \ZZ+ \ZZ \tau +\ZZ \pi \,.$$
This immediately implies that $\langle \pi+ \textbf{Fib} \rangle = \ZZ+ \ZZ \tau +\ZZ \pi$ and hence 
\[
\rank(\Lambda) = 3 \,,
\]
and hence, by Theorem~\ref{fully euc meyer}, $\Lambda$ is a Meyer set which is not fully Euclidean. 

In fact, $\Lambda$ is a regular model set. Indeed, consider 
$$
\cL:= \{ m+n \tau +k \pi : m,n,k \in \ZZ \} \subseteq \RR \times \left( \RR \times (\ZZ \pi) \right) \,.
$$
Then, it is easy to see that $(\RR, \RR \times (\ZZ \pi), \cL)$ is a CPS and 
$$
\Lambda = \oplam( [-1, \tau-1) \times \{ \pi \}) \,.
$$
\end{example}

\begin{remark}\label{rem 1} Given a CPS $(G, H, \cL)$, a window $W$ and some $a \in G$. As usual let us denote by $L=\pi_G(\cL)$.

If $a \in L$ then $a+\oplam(W)= \oplam(a^\star+W)$ is a model set in the same CPS.

Otherwise, it is shown implicitly in \cite[Prop.~5.6.19]{NS11} that one can make $a + \oplam(W)$ into a model set the following way:
  
Let $H_0$ be the cyclic subgroup of $G/L$ generated by $a+L$. Define 
$$
\cL':= \{ (x+na, na+L, x^\star) : (x,x^\star) \in \cL , n \in \ZZ \} \,.
$$
Then, $(G, H \times H_0, \cL')$ is a CPS and 
$$
a+ \oplam(W) = \oplam'( W \times \{ a+L \}) \,.
$$
Let us note here in passing that $H_0$ is a cyclic group of order at least $2$, so it is isomorphic to either some $\ZZ/n\ZZ$ or to $\ZZ$. 
\end{remark}

\begin{remark} Let $\Lambda \subseteq \RR^d$ be a Meyer set. Then, by  \cite[Cor.~5.9.20.]{NS11} and the structure theorem for compactly generated LCAG,
there exists a CPS $(\RR^d, \RR^m \times \ZZ^n \times \KK, \cL)$ with $\KK$ a compact Abelian group, and some compact $W \subseteq \RR^m \times \ZZ^n \times \KK$ such that 
$$
\Lambda \subseteq \oplam(W) \,.
$$
Now, let 
$$
\pi: \RR^m \times \ZZ^n \times \KK \to \RR^m \times \ZZ^n
$$
be the canonical projection and let 
$$
\cL':= \{ (x, \pi(x^\star)) : (x,x^\star) \in \cL \} 
$$
Then, $(\RR^d, \RR^m \times \ZZ^n, \cL')$ is a CPS and 
$$
\Lambda \subseteq \oplam(W) \subseteq \oplam'(\pi(W)) \,.
$$
This shows that every Meyer set $\Lambda \subseteq \RR^d$ is a subset of a model set in a CPS of the form $(\RR^d, \RR^m \times \ZZ^n, \cL)$, with $n \geq 0$. If the Meyer set is not
fully Euclidean, then every such CPS must have $n >0$. We suspect that the smallest possible value $n$ can take among all the CPS of this form is exactly $$n=\rank(\Lambda)-\aprank(\Lambda) \,.$$
\end{remark}

\subsection*{Acknowledgments}  The work was supported by NSERC with grant 2020-00038, and we are grateful for the support. We would also like to thank the anonymous reviewer for many helpful suggestions which improved the quality of the paper.

\end{document}